\documentclass[preprint,12pt]{elsarticle}

\usepackage{amssymb}
\usepackage{amsmath}
\usepackage{amsfonts}
\usepackage{graphicx,amsmath,amsfonts,amssymb,amscd,amsthm}
\usepackage{graphicx}
\usepackage{color}
\makeatletter
\def\EquationsBySection{\def\theequation
{\thesection.\arabic{equation}}%
\@addtoreset{equation}{section}}

\newcommand\old[1]{}
\newcommand{\pend}
{\hfill \thicklines \framebox(6.6,6.6)[l]{}}
\renewenvironment{proof}{\noindent {\it  Proof.} \rm}
{\pend}
\newtheorem{theorem}{Theorem}[section]

\newtheorem{definition}{Definition}[section]

\EquationsBySection \makeatother
\usepackage{indentfirst}

\journal{Science China Mathematics}

\begin{document}

\begin{frontmatter}



\title{Center and isochronous center of a class of
quasi-analytic switching systems} 


\author{Feng Li$^a$,\, Pei Yu$^{b,}$\footnote{Corresponding author. Tel: +1 519661-2111;
E-mail address: pyu@uwo.ca}, Yirong Liu$^c$, Yuanyuan Liu$^a$}

\address{
\small\it $^a$School of Science, Linyi University, Linyi, Shandong 276005,
P.R. China \vspace{0.0cm} \\[1.0ex] 
\small\it $^b$Department of Applied Mathematics, Western University, \\
London, Ontario N6A 5B7, Canada \\[1.0ex] 
\small\it $^c$School of Mathematics and Statistics, Central South University, \\
Changsha, Hunan 410012, P.R. China \vspace*{-0.20in} }

\begin{abstract}
In this paper, we study the integrability and linearization of a class
of quadratic quasi-analytic 
switching systems. We improve an existing
method to compute the focus values and periodic constants of
quasi-analytic switching systems. In particular, with our 
method, we demonstrate that the dynamical behaviors of quasi-analytic
switching systems are more complex than that of continuous quasi-analytic
systems, by showing the existence of six and seven limit 
cycles in the neighborhood of the origin and infinity, respectively, 
in a quadratic quasi-analytic switching system.
Moreover, explicit conditions are obtained for classifying the
centers and isochronous centers of the system.
\end{abstract}

\begin{keyword}
Quasi-analytic switching systems, Lyapunov constant,
limit cycle, center, isochronous center.

\vspace{0.10in} 
{\bf MSC(2010)} 34C07, 34C23
\end{keyword}

\end{frontmatter}



\section{Introduction}

The problem of characterizing the centers and isochronous centers of dynamical
systems has attracted the attention
of many researchers. So far, regarding the family of polynomial differential 
systems, a complete classification of centers and isochronous centers has 
only been solved for quadratic polynomial systems, or simply quadratic 
systems. Quadratic systems having a center were classified by 
Dulac \cite{Dulac1908}, Kapteyn \cite{Kapteyn1911,Kapteyn1912},
Bautin \cite{Bautin1952}, \.Zo{\l}\c{a}dek \cite{Zoladek1994}, 
Yu and Han \cite{YuHan2012}, while quadratic 
systems having an isochronous center were characterized by 
Loud \cite{Loud1964}. The centers of the cubic systems with homogeneous 
nonlinearities were studied in \cite{Vulpe1988,Loud1999},
and the isochronous centers for such cubic systems were further 
investigated by Pleshkan \cite{Pleshkan1969}. However it is still 
far away from obtaining a complete classification of the centers and 
isochronous centers for polynomial differential systems of degree three,
and it is extremely difficultly to study these problems when the degree 
of the systems is increased. For example, consider the following systems,
\begin{eqnarray}\label{Eq1.1}
\hspace{-0.60in}&& \dot{z}\!=\!(\lambda \!+\! i)z
\!+\! (\! z\bar{z})^{\frac{d-5}{2}} \! (Az^{4+j}\bar{z}^{1-j}
\!+\!\! Bz^3\bar{z}^2 \!+\! Cz^{2-\! j}\bar{z}^{3+j} \!+\!\! D\bar{z}^5\!), 
\, d \!=\! 2m \!+\! 1 \! \ge \!\! 5,
\\[0.5ex]
\label{Eq1.2}
\hspace{-0.60in}&& \dot{z} \!=\! iz+(z\bar{z})^{\frac{d-4}{2}}(Az^3\bar{z}
+Bz^2\bar{z}^2+C\bar{z}^4), \hspace{1.0in} d=2m \geq 4,
\\[0.5ex]
\label{Eq1.3}
\hspace{-0.60in}&& \dot{z} \!=\! (\lambda+i)z+(z\bar{z})^{\frac{d-3}{2}}(Az^3
+Bz^2\bar{z}+Cz\bar{z}^2+D\bar{z}^3), \quad d=2m+1\geq3,
\\[0.5ex]
\label{Eq1.4}
\hspace{-0.60in}&& \dot{z} \!=\! (\lambda+i)z+(z\bar{z})^{\frac{d-2}{2}}
(Az^2+Bz\bar{z}+C\bar{z}^2), \hspace{0.80in} d=2m \geq 2,
\end{eqnarray}
which have been investigated by Llibre and Valls
\cite{LV2009a,LV2009b,LV2009c,LV2009d},
and the conditions on centers and isochronous centers were obtained.
However, in these articles, the parameter $d$ was restricted to
make the systems be polynomial systems.

Recently, switching systems have been widely used in modelling many 
practical problems in science and engineering. Theory suggests that 
switching systems can be considered as a uniform limit of continuous systems, 
and that the global dynamics of continuous models may be approximated 
by switching systems. In fact, the richness of dynamical behavior found in
switching systems covers almost all the phenomena discussed in general 
continuous systems, such as limit cycles, homoclinic and heteroclinic orbits, 
strange attractors.  For example, Leine and Nijmeijer \cite{Leine2004}, 
and Zou and Kapper \cite{Zou2006} considered non-smooth Hopf bifurcation 
in switching systems. Bifurcation of limit cycles from centers of 
discontinuous quadratic systems was studied by 
Chen and Du \cite{Chen2010}. Limit cycles in a class of 
continuous and discontinuous cubic polynomial differential systems
were investigated by Llibre {\it et al.} \cite{Llibre2014}.
Bifurcation of limit cycles in discontinuous quadratic differential systems 
with two zones was considered in \cite{LM2014}.
The Melnikov function method has also been extended to study homoclinic 
bifurcation of non-smooth systems \cite{Du2005,Kukuka2007}. 
In addition, some general efficient methods have also been developed 
to study non-smooth 
systems. Among these methods, normal form computation for impact 
oscillators was given in \cite{Friederiksson2000}, and a general methodology
for reducing multidimensional flows to low dimensional maps in piecewise
nonlinear oscillators was proposed in \cite{Pavlovskaia2007}. 
Center and isochronous center conditions for switching systems associated 
with elementary singular points were discussed in \cite{Li2015}.

More recently, quasi-analytic systems have also been widely used in modelling
many practical problems. By ``quasi-analytic'', we mean that 
the system may be analytic for some parameters but not for some other 
parameters. For example, an axis-symmetric quasi-analytical model was 
developed in order to simulate the behavior of a RFEC system during its 
operation \cite{Musolino2012}. A quasi-analytical model  for scattering 
infrared near-field microscopy has been designed for predicting and 
analyzing signals on layered samples \cite{Hauer2012}.
A simple quasi-analytical model was developed in \cite{Oerlemans1999} 
to study the response of ice-sheets to climate. 
On the other hand, a general 
type of quasi-analytic systems, described by \vspace{-0.05in} 
\begin{equation}\label{Eq1.5}
\begin{array}{l}
\dot{x}=\delta x-y+ \displaystyle\sum_{k=2}^\infty
(x^2+y^2)^{\frac{(k-1)(\lambda-1)}{2}}X_k(x,y), \\[0.5ex]
\dot{y}=x+\delta y+ \displaystyle\sum_{k=2}^\infty
(x^2+y^2)^{\frac{(k-1)(\lambda-1)}{2}}Y_k(x,y),
\end{array}
\end{equation}

\vspace{-0.05in} 
\noindent 
where
$$
X_k(x,y)= \!\!\! \sum_{\alpha+\beta=k}A_{\alpha\beta}x^\alpha y^\beta, \qquad
Y_k(x,y)= \!\!\! \sum_{\alpha+\beta=k}B_{\alpha\beta}x^\alpha y^\beta,
$$
has been studied by Liu  \cite{Liu2002} and Liu {\it et al.} \cite{LLH2009}. 
As special cases, quadratic quasi-analytic 
systems have been studied in \cite{LL2012} and 
cubic quasi-analytic systems in \cite{P2005}.
In particular, generalized focal values and bifurcation of limit cycles 
for quadratic quasi-analytic systems were discussed in \cite{LLH2009}. 
Here, a quadratic quasi-analytic system is defined by 
taking $k=2$ only in \eqref{Eq1.5}. Similarly, cubit, etc. quasi-analytic 
systems can be defined.  

Similar to quasi-analytic continuous systems, in this paper, we propose 
to study the center and isochronous center conditions for the following class
of discontinuous planar systems: 
\begin{subequations}\label{Eq1.6}
\begin{align}
\begin{array}{l}
\dot{x}=\delta x-y+ \displaystyle\sum_{k=2}^\infty
(x^2+y^2)^{\frac{(k-1)(\lambda-1)}{2}}F_k^+(x,y), \\
\vspace{1mm}
\dot{y}=x+\delta y+ \displaystyle\sum_{k=2}^\infty
(x^2+y^2)^{\frac{(k-1)(\lambda-1)}{2}}G_k^+(x,y),
\end{array} \quad (y>0), \label{Eq1.6a} \\[-1.0ex] 
\begin{array}{l}
\dot{x}= \delta x-y+ \displaystyle\sum_{k=2}^\infty
(x^2+y^2)^{\frac{(k-1)(\lambda-1)}{2}}F_k^-(x,y), \\
\dot{y}=x+\delta y+ \displaystyle\sum_{k=2}^\infty
(x^2+y^2)^{\frac{(k-1)(\lambda-1)}{2}}G_k^-(x,y),
\end{array} \quad (y<0), \label{Eq1.6b} 
\end{align}
\end{subequations}
for $(x,y) \in {\rm R}^2$, 
where the two subsystems \eqref{Eq1.6a} 
and \eqref{Eq1.6b} describe 
dynamics on the upper and lower half planes, called 
the upper and lower systems, respectively.  
For $\lambda>0$ ($< 0$), the linear terms in \eqref{Eq1.6} are the lowest 
(highest) order terms in these functions. Hence, when $\lambda>0$,
the origin of \eqref{Eq1.6} is a center or a focus. When $\lambda<0$, 
system \eqref{Eq1.6} has no real singular point in the equator of 
Poincar\'{e} compactification, but the point at infinity is a center 
or a focus. Therefore, it is necessary to determine, for $\lambda\neq 0$,
whether or not the origin (when $\lambda >0$) or infinity 
(when $\lambda < 0$) is a center (or a weak focus ).
As a continuous work of \cite{Li2015} and \cite{Liu2002,LLH2009}, we 
generalize the study to consider center and isochronous center conditions 
of quasi-analytic switching systems and hope to promote the research 
in this direction. 

As compared with bifurcations, the center and isochronous center problems 
of switching systems have not received much attention. They should be 
considered carefully for switching systems because they are closely related 
to conditions on integrability and linearization.
Freire {\em et al.} \cite{Freire2005} discussed the center problem in a 
piecewise linear system. But it is much hard to solve center or isochronous 
center problem for nonlinear or piecewise linear systems because  the 
classical methods for computing lyapunov constants and periodic
constants are no longer applicable. Thus, new techniques are needed to 
develop.

In this paper, we study quasi-analytic switching systems mainly from 
two aspects. First of all, we modify and improve existing methods to 
compute Lyapunov constants and periodic constants for quasi-analytic 
switching systems. Secondly, as an application, we study a quadratic 
quasi-analytic switching system and derive its center and isochronous center 
conditions.

The rest of the paper is organized as follows.  In Section 2, 
we present a method to compute the return map of system \eqref{Eq1.6}. 
As an application, a class of quadratic quasi-analytic
switching systems is studied in Section 3, and the center and 
isochronous center are classified by using our method. Finally, 
our conclusion is drawn in Section 4.

\section{Lyapunov constants of the quasi-analytic 
switching system \eqref{Eq1.6}}

Under the transformation of the polar coordinates,
\begin{equation}\label{Eq2.1}
x=r^{\frac{1}{\lambda}}\cos\theta,\quad y=r^{\frac{1}{\lambda}}\sin\theta,
\end{equation}
system \eqref{Eq1.5} becomes 

\begin{equation}\label{Eq2.2}
\begin{split}
\dot{r}&=\lambda r \Big(\delta+\sum_{k=1}^\infty
\varphi_{k+2}(\theta)r^k \Big),\\
\dot{\theta}&=1+\sum_{k=1}^\infty \psi_{k+2}(\theta)r^k,
\end{split} 
\qquad (r \ge 0)
\end{equation}
where $\varphi_{k}(\theta)$ and $\psi_{k}(\theta)$ are polynomial
functions in $\cos\theta$ and $\sin\theta$, given in the form of
$$
\varphi_k(\theta)=\cos\theta X_{k-1}(\cos\theta, \sin\theta)
+\sin\theta Y_{k-1}(\cos\theta, \sin\theta),
$$
$$\psi_k(\theta)= \cos\theta Y_{k-1}(\cos\theta, \sin\theta)
-\sin\theta X_{k-1}(\cos\theta, \sin\theta).
$$
Then, it follows from system \eqref{Eq2.2} that
\begin{equation}\label{Eq2.3}
\frac{dr}{d\theta}=\lambda \, r\, \frac{\delta+\Sigma_{k=1}^\infty
\varphi_{k+2}(\theta)r^k}{1+\Sigma_{k=1}^\infty \psi_{k+2}(\theta)r^k}.
\end{equation}
Obviously, the polar coordinate form of the quasi-analytic system 
\eqref{Eq1.5} differs from
 that of analytic systems by only a constant factor $\lambda$.
It is easy to see that \eqref{Eq2.3} is a special case of the
following equation,
\begin{equation}\label{Eq2.4}
\frac{dr}{d\theta}= r\, \sum_{k=1}^\infty R_k(\theta)r^k, \quad (r \ge 0).
\end{equation}

By the method of small parameters of Poincar\'{e},
the general solution of \eqref{Eq2.3} can be expressed as \cite{Amelikin1982}
$$
r=\tilde{r}(\theta,h)=\sum_{k=1}^\infty v_k(\theta)h^k,
$$
where $v_1(0)=1,\, v_{k}(0)=0, \forall k\geq 2$.
Now, substituting the above solution $r=\tilde{r}(\theta,h)$
into \eqref{Eq2.4} yields
\begin{equation}\label{Eq2.5}
\begin{split}
v'_1(\theta)&=R_0(\theta)v_1(\theta),\\
v'_2(\theta)&=R_0(\theta)v_2(\theta)+R_1(\theta)v_1(\theta)^2,\\
&\vdots \\
v'_m(\theta)&=R_0(\theta)\Omega_{1,m}(\theta)+R_1(\theta)
\Omega_{2,m}(\theta)+\cdots+R_{m-1}(\theta)\Omega_{m,m}(\theta).
\end{split}
\end{equation}
Thus, we may solve $v_{k}(\theta)$ one by one to obtain
\begin{equation}\label{Eq2.6}
\begin{split}
v_1(\theta)&=e^{\oint_0^\vartheta R_0(\varphi)d\varphi},\\
v_2(\theta)&=2v_1(\theta) \oint_0^\vartheta R_1(\varphi)v_1(\varphi)d\varphi,\\
&\vdots \\
v_m(\theta)&=v_1(\theta) \oint_0^\vartheta \frac{R_1(\varphi)
\Omega_{2,m}(\varphi)+\cdots
+R_{m-1}(\varphi)\Omega_{m,m}(\varphi)}{v_1(\varphi)}\, d\varphi.
\end{split}
\end{equation}

Note that $R_0(\theta)=\lambda \delta$ for system \eqref{Eq2.2}.
Further, we define the successive function as
$$
\Delta(h)=\tilde{r}(2\pi,h)-h,
$$
and thus the critical point being a center must satisfy $\Delta(h)=0$, 
namely,
$$
\tilde{r}(2\pi,h)=h.
$$
Many methods have been developed to compute the successive function
$\Delta (h)$ (e.g.,see \cite{LLH2009}).

From the second equation of \eqref{Eq2.2}, we can also obtain
\begin{equation}\label{Eq2.7}
t=T(\theta,h)=\int_0^\theta \frac{d\vartheta}
{1+\sum_{k=1}^\infty \psi_{k+2}(\theta)\tilde{r}(\theta,h)^k},
\end{equation}
which implies that the critical point being an isochronous center 
should satisfy
$$
\tilde{r}(2\pi,h)= h \quad {\rm and} \quad T(2\pi,h)= 2\pi.
$$

However, the classical methods and formulas cannot
be directly applied to a non-analytic switching system
due to discontinuity. We need to modify the existing methods
to resolve this problem. 
Similar to the return map defined for analytic switching systems
\cite{GT2003}, the approach used in \cite{GT2003} (see Lemma 2.1) can be 
extended to define the return maps for the 
quasi-analytic switching system \eqref{Eq1.6}. The basic idea is briefly
illustrated as follows, see Figure~\ref{fig1}. First of all, we define 
the positive half-return map of the upper phase
of system (\ref{Eq1.6a}). Then, by a transformation $y\rightarrow -y$, 
the lower half phase could be transformed into the upper phase, 
as shown in Figure~\ref{fig2}. Further, using a time reverse changing, the 
computation of this transformed half-return map of the lower phase is 
replaced by computing the positive half-return map of the following system,
\begin{equation}\label{Eq2.8}
\begin{array}{l}
\dot{x}= \delta x-y- \displaystyle\sum_{k=2}^\infty
(x^2+y^2)^{\frac{(k-1)(\lambda-1)}{2}}F_k^-(x,-y), \\
\dot{y}=x+\delta y+ \displaystyle\sum_{k=2}^\infty
(x^2+y^2)^{\frac{(k-1)(\lambda-1)}{2}}G_k^-(x,-y),
\end{array} \quad (y>0),
\end{equation}
which is shown in Figure~\ref{fig3}. Therefore, we only need to compute 
the two positive half-return maps
for systems (\ref{Eq1.6a}) and (\ref{Eq2.8}).
\begin{figure}[!h]
\begin{center}
\vspace{2mm}
\includegraphics[height=4cm]{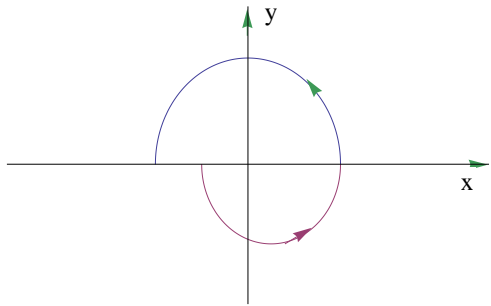}
\caption{Half-return maps for system (\ref{Eq1.6a}) and (\ref{Eq1.6b}).}
\label{fig1}
\end{center}
\end{figure}
\begin{figure}[!h]
\begin{center}
\vspace{-0.20in}
\includegraphics[height=4cm]{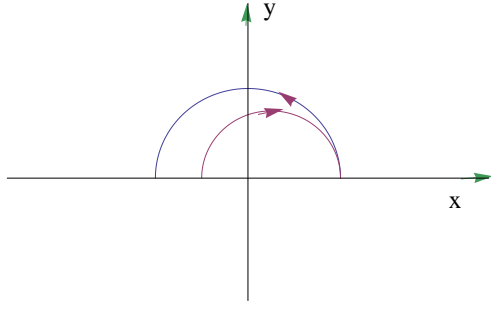}
\caption{The lower-half plane changed to the upper-half plane.}
\label{fig2}
\end{center}
\end{figure}
\begin{figure}[!h]
\begin{center}
\vspace{2mm}
\includegraphics[height=4cm]{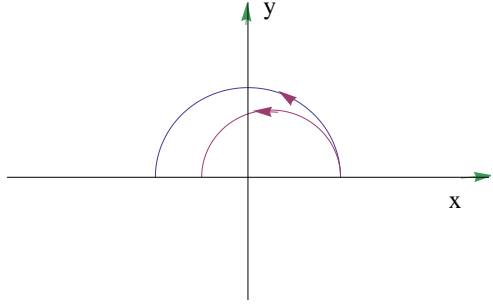}
\caption{Vector fields of systems (\ref{Eq1.6a}) and (\ref{Eq2.8})}
\label{fig3}
\end{center} 
\vspace*{-0.10in} 
\end{figure}
By defining the successive functions for systems (\ref{Eq1.6a})
and (\ref{Eq2.8}), respectively, as
$$
\Delta_1(h)=\tilde{r}_1(\pi,h)-h
\quad {\rm and} \quad
\Delta_2(h)=\tilde{r}_2(\pi,h)-h,
$$
then we obtain the successive function for the switching
system \eqref{Eq1.6}, defined as
\begin{equation}\label{Eq2.9}
\Delta(h)=\Delta_1(h)-\Delta_2(h)=\tilde{r}_1(\pi,h)-\tilde{r}_2(\pi,h).
\end{equation}

Similarly, the period constants for systems (\ref{Eq1.6a})
and (\ref{Eq2.8}) can be defined as
$$
\begin{array}{l}
T_1(\theta,h)= \displaystyle\int_0^\pi \displaystyle\frac{d\vartheta}
{1+\sum_{k=1}^\infty \psi_{2+k}(\theta)\tilde{r_1}^k(\vartheta,h)}, \\[3.0ex]
T_2(\theta,h)= \displaystyle\int_0^\pi \frac{d\vartheta}
{1+\sum_{k=1}^\infty \psi_{2+k}(\theta)\tilde{r_2}^k(\vartheta,h)},
\end{array}
$$
which in turn yield the period function for the switching
system \eqref{Eq1.6} in the form of
\begin{equation}\label{Eq2.10}
T=T_1(\pi,h)+T_2(\pi,h)=2\pi+\sum_{k=1}^{\infty} T_k h^k.
\end{equation}

In particular, if the equations describing the lower half plane are given by
\begin{equation}\label{Eq2.11}
\begin{array}{l}
\dot{x}= \delta x -y+ \displaystyle\sum_{k=2}^\infty
(x^2+y^2)^{\frac{(k-1)(\lambda-1)}{2}} F_k^-(x,y)=-y, \\
\dot{y}=x+\delta y+ \displaystyle\sum_{k=2}^\infty (x^2+y^2)^{\frac{(k-1)
(\lambda-1)}{2}} G_k^-(x,y)=x.
\end{array}
\end{equation}
then we only need to compute $\Delta_1(h)$ and $T_1(\theta,h)$.

Based on the above results, we can define the focus values and
periodic constants for the quasi-analytic switching system \eqref{Eq1.6}.

\begin{definition}
$\Delta(h)$ can be written as
$$
\Delta(h)=\sum_1^n [ u_1(\pi)-v_1(\pi) ] h^k=\sum_{k=1}^{\infty} 
V_k h^k,
$$
where $V_k$ is called the $k$th-order focus value  at the origin 
(or infinity) of the quasi-analytic switching system \eqref{Eq1.6}.
\end{definition}

\begin{definition}
$T(h)$ can be expressed as
$$
T(h)=T_1(\pi,h)+T_2(\pi,h)=2\pi+\sum_{k=1}^{\infty} T_k h^k,
$$
where $T_k$ is called the $k$th periodic constant at the origin (or infinity) 
of the quasi-analytic switching system \eqref{Eq1.6}.
\end{definition}

Having defined $V_k$ and $T_k$, we now describe the steps in computing them.
\begin{enumerate}[(1)]
\item
Introduce the transformations: $y\rightarrow -y$ and $t\rightarrow -t$
for the lower half plane.

\item
By using the transformation of polar coordinates,
$$
x=r^{\frac{1}{\lambda}}\cos\theta,\quad y=r^{\frac{1}{\lambda}}\sin\theta,
$$
in the systems (\ref{Eq1.6a}) and (\ref{Eq2.8}), write the solutions 
for the systems (2.8) and (2.10) as
$$
r_1=\tilde{r}_1(\theta,h)=\Sigma_{k=1}^\infty u_k(\theta)h^k
\quad {\rm and} \quad
r_2=\tilde{r}_2(\theta,h)=\Sigma_{k=1}^\infty v_k(\theta)h^k,
$$
respectively, satisfying $u_1(0)=v_1(0)=1,\ u_{k}(0)=v_{k}(0)=0, \
\forall k\geq 2$.

\item
Solve $u_{k}(\theta)$ and $v_{k}(\theta)$.

\item
Compute the successive function for the switching system using the formula,
$$
\Delta(h)=\Delta_1(h)-\Delta_2(h)=\tilde{r}_1(\pi,h)-\tilde{r}_2(\pi,h).
$$

\item
Compute the periodic constants for the switching system using the formula,
$$
T=T_1(\pi,h)+T_2(\pi,h).
$$
\end{enumerate}

Obviously, the symmetry principle for continuous systems is no longer
applicable for switching systems.
We need to redefine symmetry for switching systems
in order to derive the center conditions of switching systems.

\begin{definition}
If both systems (\ref{Eq1.6a}) and  (\ref{Eq1.6b})
are symmetric with respect to the $y$-axis,
then system \eqref{Eq1.6} is said to be symmetric with respect to the $y$-axis.
Further, if the vector fields of systems (\ref{Eq1.6a}) and
(\ref{Eq1.6b}) satisfy
$$
F_k^+(x,y)=-F_k^-(x,-y) \quad {\rm and} \quad G_k^+(x,y)= G_k^-(x,-y),
$$
then system \eqref{Eq1.6} is said to be symmetric with respect to the $x$-axis.
\end{definition}
So obviously, if system \eqref{Eq1.6} is symmetric with respect to 
the $x$-axis or the $y$-axis, then the origin of system \eqref{Eq1.6} 
is a center.





\section{A quadratic quasi-analytic switching system}

In this section, we consider a quadratic quasi-analytic switching system
to demonstrate the application of the formulae and results obtained in the
previous section. We will use our method to determine the center conditions
and isochronous center conditions of the system we will consider, given by 
\begin{equation}\label{Eq3.1}
\begin{array}{ll}
\begin{array}{l}
\dot{x}= \delta x -y 
+(x^2+y^2)^{\frac{(\lambda-1)}{2}}(a_{20}x^2+a_{11}xy+a_{02}y^2), 
\\[0.5ex]
\dot{y}=x + \delta y 
+(x^2+y^2)^{\frac{(\lambda-1)}{2}}(b_{20}x^2+b_{11}xy+b_{02}y^2),
\end{array} &  (y>0),\\ [3.5ex]
\begin{array}{l}
\dot{x}= \delta x -y, \\
\dot{y}=x + \delta y ,
\end{array} & (y<0).
\end{array}
\end{equation}
The case $\lambda=1$ (a polynomial system) has been studied 
in \cite{GT2003}, which becomes a special quadratic switching system.
It is shown in \cite{GT2003} that the highest order of focus values 
for this special system is $5$, and $5$ small-amplitude limit cycles 
are obtained.  
We want to extend the study to the case $\lambda \ne 1$. 
However, when the lower system is not in a simple form, 
even for general quadratic switching systems,  
it is very difficult to determine the center conditions
and isochronous center conditions. Thus, in this paper we 
focus on the study of system \eqref{Eq3.1} for $\lambda \ne 1$.

\subsection{Center conditions and limit cycles for system \eqref{Eq3.1}}

We first study the center conditions and bifurcation of limit cycles 
in system (\ref{Eq3.1}). It has been recently noticed that Tian and Yu 
studied a quadratic switching Bautin system and obtained $10$ 
small-amplitude limit cycles \cite{Yu2015}. Here, we want to show 
that system  \eqref{Eq3.1} with $ \lambda \ne 1$ can bifurcate $7$ 
limit cycles around the origin, 
two more than that of the system with $\lambda =1$. 

In order to consider the center and isochronous center conditions, 
and determine the number of limit cycles bifurcating in the small
neighborhood of the origin (or infinity), we need to compute the 
Lyapunov constants and periodic constants.  
With the aid of a computer algebra system --
Mathematica, we obtain the following Lyapunov constants of system 
\eqref{Eq3.1}.

The first three Lyapunov constants at the
origin are given by
\begin{equation}\label{Eq3.2}
\begin{split} 
L_0 & = 2 \pi \delta, \\ 
L_1&=-\frac{2}{3}(a_{11} + 2 b_{02} + b_{20})\lambda,\\
L_2&=-\frac{\pi}{8} \big[ b_{20} (a_{20} + a_{02}) 
+ ( 2 a_{20} + b_{11}) ( b_{20} + b_{02}) \big] \lambda .
\end{split}\end{equation} 
For higher Lyapunov constants, we have two cases. 

\vspace*{0.10in} 
\noindent
{\bf Case (A)} $b_{20}\neq 0$. For this case, $L_3$ is given by 
\begin{equation*}
\begin{split}
L_3 =&-\frac{2}{105} \lambda \big\{
\big[ 6 a_{20}b_{02} ( 2 a_{20}+ b_{11}) 
+ b_{20} ( 3 a_{20}^2 + 4 b_{02}^2 ) \big] (\lambda+6)  \\ 
&\qquad \qquad + 14 b_{20} \big[ 3 a_{20} ( 2 a_{20}+b_{11}) 
+ 2 b_{20} b_{02} \big] \big\}.
\end{split}
\end{equation*}
Then, there are three sub-cases in computing $L_i, \, i\ge 4$.

\vspace*{0.10in} 
\noindent 
{\emph{\bfseries Case}} {\bf (A1)} \ 
$a_{20} [  7 b_{20} + b_{02} (\lambda+6) ] \neq 0$, 
for which we have 
\begin{equation*}
\begin{split}
L_4=& \ \frac{a_{20} b_{20} \lambda \, \pi}{1536 
[ 7 b_{20} + b_{02} (\lambda+6)]^2} \big[4 b_{20}+ b_{02}(\lambda+3) \big] 
\big[ 12 b_{20} + b_{02} ( \lambda + 9) \big] \\
&\times \! \big\{ 3 a_{20}^2 ( \lambda+6)^2 
- 28 b_{20} \big[ 7 b_{20} + b_{02} (\lambda+6) \big] \big\}. 
\end{split}
\end{equation*}
\noindent 
{\bf (a)} \ If $b_{20} = -\frac{1}{4}b_{02} (\lambda+3)$, then
\begin{equation*}
\begin{split}
L_5&=-\,\frac{b_{02}\, \lambda (\lambda+3) }{3243240
(\lambda-1)^2}\, f_1,\\
L_6&=-\,\frac{\pi\, b_{02}\, \lambda^3 (\lambda+3)^2 }{5308416 a_{20}
(\lambda-1)^3}\, f_2,\\
L_7&=-\,\frac{b_{02}\, \lambda (\lambda+3)}{6788231049600 a_{20}^2
(\lambda-1)^4}\, f_3,
\end{split}\end{equation*} 
where
\begin{equation*}
\begin{split}
f_1=&\ 
(3868 a_{20}^4 \!-\! 2909 a_{20}^2 b_{02}^2 \!-\! 1085 b_{02}^4 ) \lambda^5 
\!+\! (34320 a_{20}^4 \!-\! 5962 a_{20}^2 b_{02}^2 \!-\! 
1085 b_{02}^4 ) \lambda^4  \\ 
& + 2 (41274 a_{20}^4 + 3162 a_{20}^2 b_{02}^2 + 3563 b_{02}^4) \lambda^3 
- 6 b_{02}^2 (2937 a_{20}^2 + 259 b_{02}^2) \lambda^2 \\ 
& + 243 b_{02}^2 (83 a_{20}^2 - 35 b_{02}^2) \lambda +5103 b_{02}^4 ,\\
f_2=&\ \big[ (4 a_{20}^2 -7 b_{02}^2) \lambda^2 
+ 2 (24 a_{20}^2 - 7 b_{02}^2 ) \lambda + 3( 48 a_{20}^2 + 7 b_{02}^2) 
\big]\\
& \times \! \big[ 
(196 a_{20}^4 \!+\! 25 a_{20}^2 b_{02}^2 \!+\! 4 b_{02}^4 ) \lambda^2 
\!-\! 2 b_{02}^2 (53 a_{20}^2 \!+\! 4 b_{02}^2) \lambda 
\!+\! 81 a_{20}^2 b_{02}^2 \!+\! 4 b_{02}^4 \big],\\
f_3=& \ 56 b_{02}^8 (\lambda-1)^4 (\lambda+3)
( 1027505 \lambda^4 + 426036 \lambda^3 
- 1498554 \lambda^2 \\ 
& \hspace{1.5in} +551124  \lambda +1476225 ) \\ 
&  +128 a_{20}^8 \lambda^5 ( 
9756302 \lambda^4  
+ 203092731 \lambda^3 
+ 1398704409 \lambda^2  \\ 
& \hspace{0.8in} + 3748721013 \lambda +2969005833 ) \\ 
& -4 a_{20}^6 b_{02}^2 \lambda^3 ( \lambda-1) 
( 488213167 \lambda^5 + 3344141799 \lambda^4 + 596192742 \lambda^3 \\ 
& \hspace{1.2in} 
-\! 12028807314 \lambda^2 
+ 57678458859 \lambda 
+ 102228640299 ) 
\end{split}
\end{equation*}
\begin{equation*}
\begin{split}
& +a_{20}^2 b_{02}^6 (\lambda \!-\! 1)^3 
( 24084923 \lambda^6
- 1850127480 \lambda^5 
- 6861739149 \lambda^4 \\ 
& \qquad 
+\! 1283638536 \lambda^3 
+ 11199899889 \lambda^2 
- 5507145936 \lambda 
-10721822175 
) \\ 
& -3 a_{20}^4 b_{02}^4 (\lambda \!-\! 1)^2 \lambda 
( 192786817 \lambda^6
+ 2736149472 \lambda^5 
+ 4596112377 \lambda^4 \\ 
& \quad 
-\! 12782671128 \lambda^3 
- 13724853597 \lambda^2 
- 2481265224 \lambda
-22052498589).
\end{split}
\end{equation*}

\noindent 
{\bf (b)} \ If $b_{20} = -\frac{1}{12}b_{02} (\lambda+9)$, then
\begin{equation*}
\begin{split}
L_5&=-\,\frac{b_{02}\lambda(\lambda+9)}{4169880(5\lambda+9)^3}\, \tilde{f}_1,\\
L_6&=-\,\frac{\pi\, b_{02}\, \lambda^3(\lambda+9)}
{143327232a_{20}(5\lambda+9)^4} \, \tilde{f}_2,\\
L_7&=-\, \frac{b_{02}\, \lambda(\lambda+9) }{183282238339200a_{20}^2
(5  \lambda+9)^5}\, \tilde{f}_3,
\end{split}\end{equation*}
where 
\begin{equation*}
\begin{split}
\tilde{f}_1=&
(43092 a_{20}^4 + 52305 a_{20}^2 b_{02}^2 + 8675 b_{02}^4 ) \lambda^6 
\\ 
& + (386532 a_{20}^4 +597639 a_{20}^2 b_{02}^2 + 86980 b_{02}^4 ) \lambda^5 \\
& + 9 (45900 a_{20}^4 +56538 a_{20}^2 b_{02}^2 - 8377 b_{02}^4 ) \lambda^4 \\ 
& - 162 (5202 a_{20}^4 + 68961 a_{20}^2 b_{02}^2 + 11720 b_{02}^4) \lambda^3 \\
& - 729 b_{02}^2 (47283 a_{20}^2 + 6887 b_{02}^2) \lambda^2 \\ 
& - 177147 b_{02}^2 ( 159 a_{20}^2 + 28 b_{02}^2 ) \lambda 
  - 1594323 b_{02}^4 ,\\
\tilde{f}_2=&\, \big[ (108 a_{20}^2 + 35 b_{02}^2 ) \lambda^2
+ 54 (24 a_{20}^2 + 7 b_{02}^2) \lambda 
+ 3888 a_{20}^2 + 567 b_{02}^2 \big]\\
&\times \! \big[
(11772 a_{20}^4 + 15255 a_{20}^2 b_{02}^2 + 2300 b_{02}^4) \lambda^4 \\ 
& \quad 
+ 24 (1035 a_{20}^4 + 2256 a_{20}^2 b_{02}^2 + 220 b_{02}^4) \lambda^3  \\ 
& \quad 
- 18 (2034 a_{20}^4 + 13479 a_{20}^2 b_{02}^2 + 1036 b_{02}^4) \lambda^2 \\ 
& \quad - 2592 b_{02}^2 (417 a_{20}^2 + 25 b_{02}^2 ) \lambda 
- 12393 b_{02}^2 (81 a_{20}^2 +4 b_{02}^2) \big],\\
\tilde{f}_3=&\ 279936 a_{20}^8 (\lambda-1) \lambda^5
( 38218052 \lambda^4 
+ 790049163  \lambda^3 
+  5940016749  \lambda^2 \\ 
& \hspace{1.4in} 
+ 19219084137  \lambda
+22987749699 
) \\ 
& +56 b_{02}^8 (\lambda+9) (5 \lambda+9)^4 
( 
 4362337 \lambda^5
+ 301281  \lambda^4 
- 20129886  \lambda^3 \\ 
& \hspace{1.7in}  
- 60022134  \lambda^2 
- 143311923  \lambda 
-39858075  
) \\
& +324 a_{20}^6 b_{02}^2 \lambda^3 (5 \lambda+9) 
( 
 12357727123 \lambda^6
+ 228212204300 \lambda^5 \\ 
& \qquad \qquad 
 +\! 1303142758287 \lambda^4 
+688198402656 \lambda^3 
- 17536030592607 \lambda^2 \\ 
&\qquad \qquad 
-\! 57374137182060 \lambda 
-55078300245699 
) \\ 
& +3 a_{20}^2 b_{02}^6 (5 \lambda+9)^3 
( 
 5875349809 \lambda^7
+ 75871528713 \lambda^6 
+173810927169 \lambda^5 \\ 
&\qquad \qquad  
- 698576296239 \lambda^4 
- 3526727842629 \lambda^3 
- 8234719793325 \lambda^2 
\end{split}
\end{equation*}
\begin{equation*}
\begin{split}
&\qquad \qquad  
-\! 10205502265773 \lambda 
-2605402788525 ) \\ 
& +27 a_{20}^4 b_{02}^4 \lambda (5 \lambda \!+\! 9)^2 
( 
 15822388201 \lambda^7
+ 255144264153 \lambda^6 \\ 
& \qquad \qquad  
+\! 1105548667713 \lambda^5
- 1017555100407 \lambda^4 
- 18518298666957 \lambda^3 \\ 
& \qquad \qquad  
-\! 51274399423149 \lambda^2 
- 74254559009373 \lambda 
- 56058065936181).
\end{split}
\end{equation*}

\noindent 
{\bf (c)} \ If $a_{20}^2 = \frac{28 b_{20} [7 b_{20} + b_{02}(\lambda+6)]}
{3 (\lambda+6)^2}$, the following holds: 
\begin{equation*}
\begin{split}
L_5=& \, \frac{128 b_{20}^5} {715 (\lambda+6)^4}\, 
\lambda (\lambda-1) (2 \lambda-9) (5 \lambda+16),\\
L_6= & \, 0,\\
L_7= & -\frac{128b_{20}^7}{692835 (\lambda+6)^6}\, 
\lambda(\lambda-1)(2\lambda-9) \\
& \times (1587\lambda^3 -82636\lambda^2 -527988\lambda -838080).
\end{split}
\end{equation*}

\noindent 
{\emph{\bfseries Case}} {\bf (A2)} \ 
$a_{20}=0$, under which $L_3$ is reduced to  
$$
L_3= -\, \frac{8  b_{20} b_{02}}{105} \, \lambda
\big[ 7 b_{20} + b_{02}(\lambda+6 )\big],
$$
and higher Lyapunov constants are  
\begin{equation*}
\begin{split}
L_4=&-\frac{\pi b_{02}^3 b_{11}} {3136}
\, \lambda (\lambda-1) (5 \lambda+9) ,\\
L_5=&-\frac{128 b_{02}^3} {108153045 (\lambda+6)}
\, \lambda (\lambda-1) (2 \lambda-9) \\ 
& \quad \times \big[ 9 b_{02}^2 (5 \lambda+16)(\lambda+6)^2 
+49 b_{11}^2 (27 \lambda+50) \big],\\
L_6=&-\frac{\pi b_{02}^3 b_{11}} {59006976 (\lambda+6)}
\, \lambda (\lambda-1) 
\big[ 3 b_{02}^2 (\lambda+6)^2 (562 \lambda^2 + 5085 \lambda +13365 ) \\
& \hspace{2.0in} - b_{11}^2 ( 60074 \lambda^2 - 6615 \lambda -235935 ) \big],\\
L_7=&\frac{128 b_{02}^3 }{231084662834025 (\lambda+6)^3}
\lambda ( \lambda-1) \\ 
& \times \!\!  \big[ 405 b_{02}^4 (2 \lambda -9)(\lambda+6)^4 
(1587\lambda^3 -82636\lambda^2 -527988\lambda -838080) \\ 
& + 441 b_{02}^2 b_{11}^2 (\lambda+6)^2  
(2699488 \lambda^4 +1664883 \lambda^3 -12695418 \lambda^2 \\ 
& \hspace{1.50in} +64285812 \lambda +193185000 ) \\ 
& + b_{11}^4 (17594187058 \lambda^4 + 24560280388 \lambda^3 
-109958232048 \lambda^2 \\ 
& \qquad \ 
- 10609846128 \lambda +343997992800 ) \big].
\end{split}
\end{equation*}

\noindent 
{\emph{\bfseries Case}} {\bf (A3)}  \ 
$b_{20}=-\frac{1}{7}b_{02}(6 + \lambda)$, for which $L_3$ becomes 
$$
L_3= \frac{2a_{20}^2 b_{02}}{245}\, \lambda(\lambda+6)^2,
$$ 
and higher Lyapunov constants are given by 
\begin{equation*}
\begin{split}
L_4=&-\frac{ \pi\, b_{02} \lambda (\lambda-1) (9 + 5 \lambda) } 
{43904 (6 + \lambda)}
\, \big[ a_{20} b_{02}^2 \lambda^2 
+ 2 b_{20}^2 (20 a_{20} + 7 b_{11}) \lambda \\
     &
\qquad +\!12 a_{20} ( 17 b_{02}^2- 49 a_{20}^2 - 49 a_{20} b_{11} )
+ 21 b_{11} (4 b_{02}^2 -7 a_{20} b_{11})  
\big],\\
L_5=&\frac{18 a_{20}^4 b_{02}}{105105}
\, \lambda (2 \lambda + 9) ( 7 \lambda+12) (\lambda + 2) (\lambda+6) ,\\
L_6=&\frac{\pi\, a_{20} b_{02}^3 \lambda (\lambda+6)}{4956585984}
\,\big[ 2 ( 34496 a_{20}^2 + 15375 b_{02}^2) \lambda^4 \\ 
& \qquad +\! 3 ( 29792 a_{20}^2 + 165045 b_{02}^2) \lambda^3  
- 18( 279104 a_{20}^2 + 37917 b_{02}^2) \lambda^2 \\ 
& \qquad -\! 27 (130144 a_{20}^2 + 103149 b_{02}^2) \lambda  
+1620 ( 1568 a_{20}^2 + 891 b_{02}^2)  \big],\\
L_7=&\frac{2a_{20}^2 b_{02}}{245}\, \lambda(\lambda+6)^2. 
\end{split}
\end{equation*}

\noindent 
{\bf Case (B)} \ $b_{20}=0$. For this case, we have 
\begin{equation*}
\begin{split}
L_1=& -\frac{2(a_{11} + 2 b_{02})}{3}\, \lambda, \\ 
L_2=& -\frac{ \pi b_{02}(2a_{20}+ b_{11})}{8}\, \lambda,\\
L_3=&-\frac{2 b_{02}(2a_{20} + b_{11})}{315}\, 
(8 a_{02} - 9 b_{11})\, \lambda (\lambda + 6), \\  
L_4=&-\frac{\pi b_{02}(2a_{20}\!+\! b_{11})}{9216}\, 
(28 a_{02}^2 \!+\! 36 b_{02}^2 \!-\! 36 a_{02} b_{11} \!+\! 27 b_{11}^2)
\, \lambda (\lambda+3) (\lambda+9),\\
L_5=&\ \frac{b_{02}(2a_{20} \!+\! b_{11})} {1216215}\, 
\big[ 27 b_{11}( 27 a_{02}^2 \!+\! 120 b_{02}^2 
\!-\! 60 a_{02} b_{11} \!+\! 35 b_{11}^2) \\
& \qquad \qquad \qquad - 128 a_{02} (10 a_{02}^2 \!+\! 27 b_{02}^2) \big]
\, \lambda (\lambda+2) (\lambda+12) (2 \lambda+9) ,\\
L_6=&-\frac{\pi b_{02}(2a_{20}+ b_{11})} {28311552}
\big[ 16 (143 a_{02}^4 \!+\! 594 a_{02}^2 b_{02}^2 \!+\! 243 b_{02}^4) \\ 
&\hspace*{0.5in} -\! 288 a_{02} b_{11}( 11 a_{02}^2 \!+\! 45 b_{02}^2) 
\!+\! 1080 b_{11}^2 (3 a_{02}^2 \!+\! 7 b_{02}^2 \\
&\hspace*{0.5in} -\! 45 b_{11}^3 ( 56 a_{02} \!-\! 27 b_{11}) \big]
 \lambda (\lambda+3) (\lambda+6) (\lambda+15) (2 \lambda+3).
\end{split}
\end{equation*}
Note that in the above computations, 
$L_{k-1}=0$, $k=1,2, \cdots,6$ have been used in computing $L_k$.

Now, by carefully analyzing the above Lyapunov constants, we obtain 
the following result. 

\begin{theorem}\label{Thm3.1} 
For system \eqref{Eq3.1}, maximal six small-amplitude 
limit cycles can bifurcate from the origin and maximal 
seven small-amplitude 
limit cycles can exist in the neighborhood of infinity. 
Moreover, the first seven Lyapunov constants at the origin (or infinity) 
of system \eqref{Eq3.1} vanish if and only if one of the following 
conditions is satisfied: 

{\rm (i)} \hspace{0.10in} $\delta = a_{11}=b_{20}=b_{02}=0$, 

{\rm (ii)} \hspace{0.05in} 
$\delta = b_{20}= a_{11} +2b_{02} = b_{11} +2a_{20}=0$, 

{\rm (iii)} \hspace{0.00in} 
$\delta = a_{20}=b_{02}= a_{11} +b_{20}= b_{11} +a_{02} = 0$, 

{\rm (iv)} \hspace{0.00in} 
$\delta = \lambda-1 = a_{02}=a_{20}= b_{02} +b_{20} = a_{11}-b_{20} = 0$,

\vspace*{0.05in} 
{\rm (v)} \hspace{0.03in} 
$ 
\left\{\! \begin{array}{l}
\delta = \lambda-1 
= 2a_{11}b_{20}+3a_{20}^2-2b_{20}^2=2b_{11}+5a_{20}=0,\\[0.5ex]
8a_{02}b_{20}^2+a_{20}(8b_{20}^2-3a_{20}^2)=4b_{02}b_{20}
-3a_{20}^2+4b_{20}^2=0,
\end{array} 
\right. 
$

\vspace*{0.05in} 
{\rm (vi)} \hspace{0.00in} 
$\left\{\! \begin{array}{l}
\delta = \lambda-\frac{9}{2} = 12a_{11}b_{20}+27a_{20}^2 
-4b_{20}^2=4b_{11}+11a_{20}=0,\\[0.5ex]
32a_{02}b_{20}^2+3a_{20}(8b_{20}^2-9a_{20}^2)=24b_{02}b_{20}
-27a_{20}^2+16b_{20}^2=0.
\end{array}
\right.
$ 
\end{theorem}

\begin{proof} 
First of all, note that $\delta = 0$ is a necessary condition for all cases
in order to get limit cycles bifurcating from the origin (or infinity), 
under which $ L_0 = 0$. 

We start from Case (B) in which $b_{20}=0$. 
It is easy to see that for this 
case all $L_i, \, i=2,3, \dots,6 $ contains a same factor 
$ b_{02} ( 2 a_{20} + b_{11}) $. Thus,  
$L_i = 0, \, i=2,3, \dots 6 $, if $ b_{02} ( 2 a_{20} + b_{11}) = 0$, 
implying that the maximal limit cycles can be obtained is two.
When $b_{20} = b_{02} =0$, $L_2 =0$, and $L_1 = 0$ yields one solution: 
$a_{11} = 0$, which gives the condition (i). 
If $b_{20} = 2 a_{20} + b_{11} =0 $, then $L_2 = 0$, and 
$L_1 = 0$ requires $ a_{11} + 2 b_{02} = 0$, which 
yields the condition (ii). 

The Case (A3) in which $ b_{20} = -\frac{1}{7} b_{02} (6 + \lambda)$ is 
simple since it is assumed $b_{20} \ne 0$ and so $b_{02}$. 
It is also assumed $a_{20} \ne 0$ for this case, yielding 
$L_3 \ne 0$, implying that three limit cycles can be obtained 
since one can choose appropriate values of $ a_{11}$ and $b_{11}$ 
to set $ L_1 = L_2 = 0$. 

Next, consider Case (A2) in which $a_{20}=0$. 
Maximal $5$ limit cycles may be obtained by choosing $ b_{02} \ne 0$, and 
setting  $b_{11} = 0$ (so $L_4 = 0$), 
$b_{20} = 0$ (so $L_2 = L_3 =0$), and 
$a_{11} + 2 b_{02} = 0$ (so $L_1 = 0$).  
If $b_{02} =0$, then $L_3 = L_4 = \cdots = L_7 =0$. 
Further, setting $L_1=L_2 =0$ we obtain $a_{11}+ b_{20} 
= b_{11} + a_{02} =0$, which is the condition (iii). 
Another possibility for this case is to set $\lambda=1$, giving 
$L_4 = L_5 = L_6 = L_7 = 0$. Further, letting 
$L_1 = L_2 = L_3 = 0$ yields 
$a_{02} = b_{02} + b_{20} = a_{11} - b_{20} = 0$, leading to 
the condition (iv). 
 
For Case (A1)(c), it is noted that 
$\lambda = 1 $ or $\lambda = \frac{9}{2}$ yields 
$ L_5 = L_7 = 0$ ($L_6$ is already equal to zero). 
$ L_1 = 0$ gives $ b_{02} = -\frac{1}{2}(a_{11} + b_{20})$ which
is substituted into $a_{20}^2 = \frac{28 b_{20} [7 
b_{20} + b_{02} ( \lambda + 6) ] }{ 3 (\lambda+6)^2} $ to obtain 
$ 2 a_{11} b_{20} + 3 a_{20}^2 - 2 b_{20}^2 = 0$ for $\lambda=1$, 
and $12 a_{11} b_{20} + 27 a_{20}^2 - 4 b_{20}^2 = 0$  
for $ \lambda = \frac{9}{2}$.  
Then, for case (v) for which $\lambda = 1$, 
using $a_{20}^2 = \frac{4}{3} b_{20}(b_{20}+b_{02}) $ 
to simplify $L_3 = 0$ yields $ 2 b_{11} + 5 a_{20} = 0$, and 
again using the expression of $a_{20}^2$ as well as 
$b_{11} = -\frac{5}{2} a_{20}$ to simplify $L_2 = 0$ we obtain 
$ 8 a_{02} + a_{20}( 8 b_{20}^2 - 3 a_{20}^2)=0$. Finally, using 
$a_{11} = \frac{2 b_{20}^2 - 3 a_{20}^2}{ 2 b_{20}} $ to simplify 
$L_1 = 0$ results in $ 4 b_{02} b_{20} - 3 a_{20}^2 + 4 b_{20}^2=0$. 
Summarizing the above results leads to the condition (v). 
Following the same procedure, we can obtain the condition (vi). 

Now we come to Cases (A1)(a) and (A1)(b). 
First note that for these two cases, $b_{02} \ne 0$ since it is assumed
$ b_{20} \ne 0$. Therefore, instead of considering the equations 
$L_5 = L_6 = L_7 = 0$, 
we consider the polynomial equations $ f_1 = f_2 = f_3 = 0$ 
for Case (A1)(a) and $ \tilde{f}_1 = \tilde{f}_2 = \tilde{f}_3 = 0$ 
for Case (A1)(b). 
In order to obtain the maximal number of limit cycles, we need to find 
the conditions such that $L_i=0, i=0,1, \dots, k-1$, but $L_k \ne 0$. 
Actually,  setting $ L_i=0, \, i=0,1,2,3,4$, we obtain 
two sets of solutions: 
\begin{equation}\label{Eq3.15}
\begin{split}
\delta & = 0, \\ 
b_{20}&= - \, \frac{b_{02}}{4}\, (\lambda+3), \\ 
a_{11}&= \frac{b_{02}}{4}\, (\lambda-5),\\
a_{02}& = \frac{ a_{20}^2\, \lambda - b_{02}^2 (\lambda-1) }{6a_{20}},\\
b_{11}& = -\, \frac{ a_{20}^2 (\lambda^2 + 21 \lambda + 6) 
- b_{02}^2 (\lambda-1)(\lambda+3) }{6 a_{20} (\lambda-1)},
\end{split}
\end{equation} 
for Case (A1)(a), and 
\begin{equation}\label{Eq3.16}
\begin{split}
\delta & = 0, \\ 
b_{20} &=-\,\frac{b_{02}}{12}\, (\lambda+9), \\ 
a_{11} & = \frac{b_{02}}{12}\, (\lambda-15), \\
a_{02} &= -\,\frac{ 9 a_{20}^2 \, \lambda \, ( \lambda+13) 
                   + b_{02}^2 (\lambda-3) (5 \lambda+9) }
{18 a_{20} (5 \lambda+9)},\\
b_{11}&= \frac{ a_{20}^2 (9 \lambda^2 - 45 \lambda + 162) 
+ b_{02}^2 (\lambda+9)(5 \lambda+9) }{18 a_{20} (5 \lambda+9)},
\end{split}
\end{equation}
for Case (A1)(b). 

Note that the functions $f_1,\,f_2,\,f_3$ and $\tilde{f}_1,\,\tilde{f}_2,\,
\tilde{f}_3$ are homogeneous polynomials in $a_{20}^2$ and $b_{02}^2$. 
So we may introduce $b_{02}^2 = k\, a_{20}^2$ ($k>0$) 
into these polynomials and note that $f_2$ and $\tilde{f}_2$ 
have two factors, one is linear in $k$ and one is quadratic in $k$, 
given as follows: 
\begin{equation}\label{newEq2} 
\begin{array}{ll} 
f_2 = -\,a_{20}^6 f_{2a} f_{2b} = \!\!\!\! & -\,a_{20}^6 
\big[ 7 (\lambda-1)(\lambda +3) k - 4 (\lambda + 6)^2  \big] \\ [0.5ex] 
\!\!\!\! & \hspace{0.10in} \times \! \big[ 
4 (\lambda-1)^2 k^2 + (\lambda-1)(25 \lambda -81) k +196 \lambda^2 \big].
\end{array} 
\end{equation} 
We first solve $f_{2a}=0$ to obtain 
$ k= \frac{4 (\lambda + 6)^2}{7 (\lambda-1)(\lambda +3)}$ 
which is then substituted into $f_1$ and $f_3$ to yield 
$$ 
\begin{array}{ll} 
f_1 = \!\!\!\! & -\, \frac{1296 a_{20}^4}{7} (\lambda -1)(2 \lambda - 9) 
(5 \lambda + 16) ( \lambda+3)^2, \\[0.5ex] 
f_3 = \!\!\!\! & -\frac{699840 a_{20}^8}{49}(\lambda -1)^2(2 \lambda - 9) 
( \lambda + 3)^3 \, C_3  
\end{array} 
$$ 
where 
\begin{equation}\label{newEq1} 
C_3 = (1587 \lambda^3 - 82636 \lambda^2 -  527988 \lambda - 838080). 
\end{equation} 

Since $(\lambda-1)(\lambda+3) \ne 0$, the only solution satisfies 
$f_1 = 0$ and $ f_3 \ne 0$ is $ \lambda = -\frac{16}{5}$. When 
$\lambda = \frac{9}{2}$, $L_i=0,\, i=0,1, \dots, 7$, but it is easy 
to verify that this is a special case of (vi). 

Moreover, for the solution $ \lambda = -\frac{16}{5}$, 
$$
\det(J)= \det \left[ \frac{\partial (L_5, L_6)}
{\partial(k, \lambda)} \right]  
= \frac{8384103915264}{78125}\, a_{20}^{10} \ne 0 \quad 
{\rm for} \ \, a_{20} \ne 0, 
$$ 
implying that seven limit cycles can bifurcate in the small neighborhood 
of infinity. 

For the second factor $f_{2b}$, we eliminate $k$ from the two equations 
$ f_1 = f_{2b} = 0$ to obtain the solution for $k$: 
$$ 
k = -\, \frac{4 \lambda^2 (6337 \lambda^2+27872 \lambda+27783)}
         {(\lambda-1)( 1721 \lambda^3 - 2945 \lambda^2 - 38097 \lambda-45927}, 
$$ 
and a resultant equation,  
$$
\begin{array}{ll}  
{\rm R_{12}} = \!\!\!\! & \lambda (\lambda +1) (\lambda -1)
      (397378 \lambda^4+3696797 \lambda^3+12760835 \lambda^2 \\[0.5ex] 
\!\!\!\!& \hspace{1.10in} +19311435 \lambda+10762227) = 0 ,  
\end{array}
$$ 
which has three real solutions: $\lambda = -2.59473685 \cdots,\, 
-1.58363608 \cdots,\, -1$,\, but none of them yields $k>0$. 
Hence, there are no solutions from $f_{2b}=0$ to generate seven limit cycles.

When the conditions in (\ref{Eq3.16}) are satisfied, similarly, we 
can use $ b_{02}^2 = \tilde{k} \, a_{20}^2 $ to find that 
\begin{equation}\label{newEq3} 
\begin{array}{rl} 
\tilde{f}_2= \!\!\!\! & a_{20}^6 \tilde{f}_{2a} \tilde{f}_{2b} \\[0.5ex] 
= \!\!\!\!  &  a_{20}^6 \big[ 7 (\lambda+9) (5 \lambda+9) \tilde{k} + 108 (\lambda+6)^2 \big] \big[ 4 (5 \lambda+9)^2 (23 \lambda^2-30 \lambda-153) 
\tilde{k}^2 \\[0.5ex] 
\!\!\!\! & 
+9 (5 \lambda L \!+\! 9) (339 \lambda^3 \!+\! 593 \lambda^2 
\!-\! 6459 \lambda \!-\! 12393) \tilde{k}
\!+\! 108 \lambda^2 (\lambda \!-\! 1) (109 \lambda \!+\! 339) \big].
\end{array}
\end{equation} 
Solving $\tilde{f}_{2a}=0$ to obtain 
$\tilde{k} = -\frac{108 (\lambda+6)^2}{7 (\lambda+9) (5 \lambda+9)}$, 
and then substituting it into $\tilde{f}_1$ and $\tilde{f}_3$ yields 
$$ 
\begin{array}{ll} 
\tilde{f}_1 = \!\!\!\! & -\frac{34992 a_{20}^4}{49} (\lambda -1)(2 \lambda - 9) 
(5 \lambda + 16) ( \lambda+9)^2, \\[0.5ex] 
\tilde{f}_3 
= \!\!\!\! & -\frac{17061120 a_{20}^8}{49}(\lambda -1)^2(2 \lambda - 9) 
(5 \lambda + 9)^2 ( \lambda + 9)^3 \, C_3,
\end{array} 
$$  
where $C_3$ is given in \eqref{newEq1}. 
Since $(5 \lambda+9)(\lambda+9) \ne 0$, the only solution satisfies 
$f_1 = 0$ and $ f_3 \ne 0$ is $ \lambda = -\frac{16}{5}$. 
$\lambda =1 $ and $\lambda = \frac{9}{2}$ are not solutions 
since they yield $ \tilde{k}<0$. 
For the solution $ \lambda = -\frac{16}{5}$, we have 
$$
\det(J)= \det \left[ \frac{\partial(L_5,L_6)}
{\partial(\tilde{k}, \lambda)} \right] 
= -\,\frac{ 304277047914997479168}{78125}\, a_{20}^{10} \ne 0 \quad 
{\rm for} \ \, a_{20} \ne 0, 
$$ 
implying that seven limit cycles can bifurcate in the small neighborhood 
of infinity. 

For the second factor $\tilde{f}_{2b}$, similarly we eliminate $\tilde{k}$ 
from the two equations $\tilde{f}_1=0$ and $\tilde{f}_{2b}=0$ to
obtain the solution for $\tilde{k}$: 
$$ 
\tilde{k}= - \textstyle\frac{36 \lambda^2 (\lambda-1)
          (223 \lambda^4+8120 \lambda^3-127950 \lambda^2-778680 \lambda
           -741393)} {(5 \lambda+9)
          /(6419 \lambda^6+31564 \lambda^5-1363197 \lambda^4
          -4806072 \lambda^3+20436705 \lambda^2 +88888428 \lambda+81310473)},
$$ 
and a resultant equation, 
$$ 
\begin{array}{ll} 
{\rm \tilde{R}_{12}} 
= \!\!\!\! &  \lambda (\lambda -1) (\lambda -9) (5 \lambda +9)
(8397602 \lambda^9+84616511 \lambda^8-1494342124 \lambda^7 \\[0.5ex] 
\!\!\!\! & -16706405616 \lambda^6 +16325397720 \lambda^5+375950483190 \lambda^4
\\[0.5ex] 
\!\!\!\! & -\,756410507892 \lambda^3 -10202463072792 \lambda^2
-22941003813786 \lambda \\[0.5ex] 
\!\!\!\! & -\, 15945065065773), 
\end{array}  
$$ 
which has four real solutions: $\lambda = -12.38286360 \cdots,\, 
-8.54108488 \cdots,\,1,\, 9$, $13.13951523 \cdots$, 
but all of them yield $\tilde{k} \le 0$.
Thus, there are no solutions from $\tilde{f}_{2b}=0$ to give 
seven limit cycles. 

Summarizing the above results obtained for Cases (A1)(a) and (A1)(b), 
we conclude that there exist two sets of infinite solutions 
such that system \eqref{Eq3.1} 
can have seven limit cycles bifurcating in the small neighborhood of 
infinity. 

Although we can not obtain seven limit cycles around the origin of 
system \eqref{Eq3.1}, we may find an infinite number of 
solutions for six limit cycles which 
bifurcate in the small neighborhood of the origin, which is still better 
than the five limit cycles obtained in~\cite{GT2003}. 
To find the solutions, it needs $L_i,\, i=0,1, \cdots, 5$, but 
$L_6 \ne 0$. Thus, we only need to solve $ f_1 = 0$ (or $\tilde{f}_1 =0$).  
Letting $b_{02}^2 = k a_{20}^2$, $f_1 = 0 $ becomes 
$ f_1 = -\, a_{20}^2 [ A_2 k^2 + A_1 k + A_0 ]$, where 
$$
\begin{array}{ll}  
A_2 = 7 (\lambda - 1)^2 (\lambda+3) (155 \lambda^2-243), \\[0.5ex] 
A_1 = \lambda (\lambda-1) (2909 \lambda^3+8871 \lambda^2+2547 \lambda+20169), 
\\[0.5ex] 
A_0 = -4 \lambda^3 (967 \lambda^2+8580 \lambda +20637).   
\end{array}
$$  
We want to find solutions satisfying $k>0$ and $ \lambda >0 \ (\lambda \ne 1)$. 
It is easy to show that $f_1 = 0$ has a unique positive solution 
for $k $ when $ \lambda \ge 9\sqrt{3/155}$, and 
does not have solutions when $ 1 < \lambda < 9\sqrt{3/155}$. 
When $ \lambda \in (0,1)$, $A_i < 0, \, i=0,1,2$, and 
$$ 
\begin{array}{ll} 
\Delta = A_1^2 - 4 A_2 A_0 = \!\!\!& 
-81 \lambda^2 (1-\lambda)^3 (311721 \lambda^5+3409519 \lambda^4
+14178654 \lambda^3 \\[0.5ex] 
\!\!\! & +25596434 \lambda^2+14511609 \lambda-5022081),  
\end{array}
$$
which is positive for $\lambda \in (0,\,0.23512585)$, leading to 
that $f_1=0$ has two solutions for $k$ for each $\lambda$ chosen 
from this interval.  
Hence, there exists an infinite number of solutions for 
$k$ satisfying $ f_1 =0$ when $ \lambda \in (0,\,0.23512585) 
\cup (9\sqrt{3/155},\, \infty)$.  
These solutions do not yield $f_2 =0$, since in the above we have already 
shown that $f_1 = f_2 =0$ does not have solutions satisfying 
$k>0$ and $\lambda >0$. 
This indicates that for Case (A1)(a) there exists an infinite 
number of solutions for the existence of six limit cycles around the 
origin of system \eqref{Eq3.1}. Similarly, for Case (A1)(b),
we can prove that $f_1 = 0$ has two positive solutions for $k$ when 
$ \lambda \in (0,\,1)$ and one positive solution when 
$\lambda \in (1,\, 5.132341426)$, implying that for Case (A1)(b)
there also exists an infinite number of solutions 
for the existence of six limit cycles around the
origin of system \eqref{Eq3.1}. 

The proof is complete.
\end{proof} 

\vspace{0.10in} 
Note that the conditions (i)--(vi) given in Theorem~\ref{Thm3.1} yield 
$L_i=0,\, i=0,1, \dots, 7$, implying that they are necessary conditions 
for the origin (or infinity) of system \eqref{Eq3.1} to be a center.  
In the following, we will show that these conditions are also sufficient 
for the origin (or infinity) of system \eqref{Eq3.1} to be a center.
We have the following theorem. 

\begin{theorem}\label{Thm3.2} 
The conditions {\rm (i)--(vi)} given in Theorem \ref{Thm3.1} 
are necessary and sufficient for the origin (or infinity) 
of system \eqref{Eq3.1} to be a center.
\end{theorem} 

\begin{proof} 
The {\it necessity} has been shown in the proof of Theorem~\ref{Thm3.1}. 
Hence, we only need to prove the {\it sufficiency}.
First note that for all the six cases, the lower-half plane is same 
(as $\delta = 0$), described by
\begin{equation*}
\begin{array}{l}
\dot{x}=-y, \\
\dot{y}=x,
\end{array}
\qquad (y<0).
\end{equation*}
This system has a first integral $H_0(x,y) = x^2 + y^2$, 
which is an even function of $x$ (i.e., symmetric with the $y$-axis). 
Thus, in the following, for each case we only list the equations for
the upper-half plane.

When the condition (i) holds,
the equations for the upper-half plane of
system (\ref{Eq3.1}) become
\begin{equation}\label{Eq3.9}
\begin{array}{l}
\dot{x}=-y+ (a_{20}x^2+a_{02}y^2) (x^2+y^2)^{\frac{(\lambda-1)}{2}} , \\[0.5ex]
\dot{y}=x+ b_{11}xy\, (x^2+y^2)^{\frac{(\lambda-1)}{2}} ,
\end{array}
\qquad (y>0), \\
\end{equation}
which is symmetric with the $y$-axis, and so 
the origin (or infinity) is a center.

When the condition (ii) is satisfied, the equations for the
upper-half plane of system (\ref{Eq3.1}) can be rewritten as
\begin{equation}\label{Eq3.10}
\begin{array}{l}
\dot{x}=-y+ (a_{20}x^2-2b_{02}xy+a_{02}y^2) (x^2+y^2)^{\frac{(\lambda-1)}{2}},
\\[0.5ex]
\dot{y}=x- (2a_{20}x-b_{02}y)y\, (x^2+y^2)^{\frac{(\lambda-1)}{2}} ,
\end{array}
\qquad (y>0),
\end{equation}
which has an integrating factor
$
\frac{2}{3}\lambda(x^2+y^2)^{\frac{(1-\lambda)}{2}}.
$
Then, system \eqref{Eq3.10} becomes 
$$ 
\begin{array}{l}
\dot{x}=-\frac{2}{3}\lambda y(x^2+y^2)^{\frac{(1-\lambda)}{2}}
+\frac{2}{3}\lambda(a_{20}x^2-2b_{02}xy+a_{02}y^2),
\\[0.5ex]
\dot{y}=\frac{2}{3}\lambda x(x^2+y^2)^{\frac{(1-\lambda)}{2}}
-\frac{2}{3}\lambda(2a_{20}x-b_{02}y)y ),
\end{array}
$$
which has a first integral,
$$
\begin{array}{ll}
H_1(x,y)= &\!\!\! 6 \lambda (x^2 + y^2)^{\frac{(3-\lambda)}{2}}  
+ 2 \lambda (\lambda-3)\, y 
\big[ 3 (a_{20}x - b_{02}y)x  + a_{02}y^2 ) \big]. 
\end{array}
$$ 
It is seen that $H_1$ is an even function of $x$ when $y=0$, and so
the origin (or infinity) of system \eqref{Eq3.1} is 
a center (e.g., see Theorem 2.2 in~\cite{Li2015}). 

When the condition (iii) holds, the equations for the upper-half
plane of system (\ref{Eq3.1}) can be rewritten as
\begin{equation}\label{Eq3.11}
\begin{array}{l}
\dot{x}=-y - (b_{20}x-a_{02}y) y\, (x^2+y^2)^{\frac{(\lambda-1)}{2}} , 
\\[0.5ex]
\dot{y}=x + (b_{20}x-a_{02}y) x \, (x^2+y^2)^{\frac{(\lambda-1)}{2}} ,
\end{array}
\qquad (y>0).
\end{equation}
It is easy to see that system \eqref{Eq3.11} has a first integral,
$$
H_2(x,y)=(x^2 + y^2)^{\frac{\lambda}{3}},
$$
which is an even function of $x$ and so
the origin (or infinity) is a center.

When the condition (iv) is satisfied, the equations for the
upper-half plane of system \eqref{Eq3.1} become
\begin{equation}\label{Eq3.12}
\begin{array}{l}
\dot{x}=-y+b_{20}xy, \\[0.5ex]
\dot{y}=x+(b_{20}x^2+b_{11}xy-b_{20}y^2),
\end{array}
\qquad (y>0). 
\end{equation}
It can be shown that system \eqref{Eq3.12} has a first integral,
$$
H_3(x,y)= (b_{20}x-1)\big[ b_{20}x+1+12(b_{11}-\gamma)y \big]^\alpha
\big[ b_{20}x+1+12(b_{11}+\gamma)y \big]^{(1-\alpha)},
$$
where
$$
\alpha= \frac{4b_{20}^2}{ \gamma(\gamma+b_{11})}, 
\quad \gamma=\sqrt{b_{11}^2+8b_{20}^2}.
$$
$H_3(x,y)$ is an even function of $x$ when $y=0$ because
$
H_3(x,0) =b_{20}^2x^2-1,
$
and so the origin is a center~\cite{Li2015}. 

When the condition (v) holds,
the upper-half plane has a first integral,
$$
H_4(x,y) \!=\! (2b_{20}x-a_{20}y-2)^2 \big[ 4(b_{20}x + 1)^2
-(4a_{20}+12a_{20} b_{20}x)y+(3a_{20}^2-8b_{20}^2)y^2  \big],
$$
which is an even function of $x$ when $y=0$ since 
$
H_4(x,0)=16(1-b_{20}^2x^2)^2.
$
So the origin is a center~\cite{Li2015}. 

Finally, when the condition (vi) is satisfied, the equations
for the upper-half plane of system \eqref{Eq3.1} become
\begin{equation}\label{Eq3.13}
\begin{split}
\dot{x}= &-y+ \frac{1}{96 b_{20}^2}(x^2 + y^2)^{\frac{7}{4}}
\big[ 96 a_{20} b_{20}^2 x^2 - 8 b_{20} ( 27 a_{20}^2 - 4 b_{20}^2) x y\\
&\hspace{1.5in}+ 9 a_{20} (9 a_{20}^2 \!-\! 8  b_{20}^2) y^2 \big],\\
\dot{y}=&x+\frac{1}{24\, b_{20}} (x^2+y^2)^{\frac{7}{4}}
\big[ 24 b_{20}^2 x^2 - 66 a_{20} b_{20}xy + (27 a_{20}^2 - 16 b_{20}^2 )y^2  
\big],
\end{split}
\end{equation}
which has a first integral,
\begin{equation*}
\begin{split}
H_5(x,y)= & \frac{9 a_{20}^2}{64b_{20}^2 
(9a_{20}^2+ 16b_{20}^2)(x^2+y^2)^{\frac{3}{2}}}
\Big\{\! 4096 b_{20}^4 \\ 
& + (4 b_{20} x \!-\!3 a_{20} y)^4 
\big[ 8 b_{20}^2 (x^2 \!-\! y^2) 
\!-\! 3 a_{20} (8 b_{20} x \!-\! 3 a_{20} y) y \big] 
(x^2 \!+\! y^2)^{\frac{3}{2}}  \\
& + 128 b_{20}^2 y (4 b_{20} x \!-\! 3 a_{20} y)
(24 a_{20} b_{20} x \!-\! 9 a_{20}^2 y \!+\! 16 b_{20}^2 y) 
(x^2 + y^2)^{\frac{3}{4}} \Big\}.
\end{split}
\end{equation*}
Since 
$
H_5(x,0)=-\frac{288 a_{20}^2 b_{20}^2 (2 +b_{20}^2 x^6 
|x|^3)}{(9 a_{20}^2 + 16b_{20}^2) |x|^3},  
$
$H_5(x,y)$ is an even function of $x$ when $y=0$, 
and hence, the origin is a center~\cite{Li2015}. 
\end{proof}

Combining the results in Theorems~\ref{Thm3.1} and \ref{Thm3.2}, we have 

\begin{theorem}\label{Thm3.3}
For system \eqref{Eq3.1}, the highest order of focus value is $7$.
\end{theorem}

\subsection{Isochronous centers of system \eqref{Eq3.1}}

Having established the center conditions in the previous section for
system (\ref{Eq3.1}), we now discuss the isochronous center conditions for
this system.

First, by a direct computation, we can show that
under any of the conditions (ii)--(vi), 
no isochronous center can exist because these conditions cannot lead to
all periodic constants vanishing. When the condition (i) holds,
the periodic constants are obtained as
$$
\begin{array}{ll}
\tau_1= \displaystyle\frac{2}{3}\,(a_{20} + 2 a_{02}-b_{11}), \\[1.8ex]
\tau_2= \displaystyle\frac{\pi}{16}a_{02}^2 \big[ 2a_{20}\lambda 
+ a_{02} (3 \lambda+5) \big] ,
\\ [1.8ex]
\tau_3= -\, \displaystyle\frac{2a_{02}^3}{105} \, (\lambda^2-1), \\ [1.8ex]
\tau_4= \displaystyle\frac{3\,\pi a_{02}^4 }{1024} 
\, (\lambda^2-1) (\lambda+3), 
\\[1.8ex]
\tau_5= -\, \displaystyle\frac{2 a_{02}^5}{1001} \, (\lambda^2-1)  
(14 \lambda^2 + 91\lambda + 139 ),
\\[1.8ex]
\tau_6= \displaystyle\frac{3\,\pi a_{02}^6}{4587521024}\, (\lambda^2-1) 
(1190\lambda^2 + 8817\lambda +13898 ),
\\[1.0ex]
\quad \ \vdots
\end{array}
$$
Therefore, we have the following theorem.

\begin{theorem}\label{Thm3.4}
The origin (or infinity) of system \eqref{Eq3.1} is an isochronous center 
if and only if one of the following conditions holds:
$$
\begin{array}{cl}
{\rm (I)} & a_{11}=b_{02}=b_{20}=a_{02}= b_{11}-a_{20}=0; \\[0.5ex]
{\rm (II)} & \lambda-1= a_{11}=b_{02}=b_{20}= b_{11} +2a_{02} 
= a_{20} +4a_{02} = 0;
\\[0.5ex]
{\rm (III)} & \lambda +1= a_{11}=b_{02}=b_{20}= b_{11}-3a_{02} 
= a_{20}-a_{02} = 0.
\end{array}
$$
\end{theorem}

\begin{proof}
The {\it necessity} can be easily proved by setting 
the periodic constants $\tau_i = 0, \, i=1,2, \cdots 6$. 
To prove the {\it sufficiency}, we consider the three 
systems under the three conditions (I), (II) and (III). 
First, consider condition (I). When this condition holds,  
the equations for the upper-half plane
of system \eqref{Eq3.1} can be written as
\begin{equation}\label{Eq3.17}
\begin{array}{l}
\dot{x}=-y+a_{20}x^2(x^2+y^2)^{\frac{(\lambda-1)}{2}}, \\[0.5ex]
\dot{y}=x+a_{20}xy(x^2+y^2)^{\frac{(\lambda-1)}{2}}.
\end{array}\end{equation}
A simple calculation gives $\frac{d\theta}{dt}=1$, and so the origin 
(or infinity) of system \eqref{Eq3.1} is an isochronous center. 

When the condition (II) is satisfied, system \eqref{Eq3.1} becomes
\begin{equation}\label{Eq3.18}
\begin{array}{ll}
\begin{array}{l}
\dot{x}=-y - a_{02} (4 x^2-y^2), \\[0.5ex]
\dot{y}=x-2a_{02}xy,
\end{array}
& (y>0), \\[2.0ex]
\begin{array}{l}
\dot{x} = -y, \\ [0.0ex]
\dot{y} = x,
\end{array}
& (y<0).
\end{array}
\end{equation}
This system has a transversal system
\begin{equation}\label{Eq3.19}
\begin{array}{ll}
\begin{array}{l}
\dot{x} =x- 4 a_{02}\, xy \, (1-a_{02}y), \\[0.5ex]
\dot{y} =y- a_{02}\,y^2 (3-2a_{02}y),
\end{array}
& (y>0), \\[2.0ex]
\begin{array}{l}
\dot{x}=x, \\[0.0ex]
\dot{y}=y,
\end{array}
& (y<0).
\end{array}
\end{equation}
So by Theorem 2.1 in \cite{Sabatini1992}, the origin of system \eqref{Eq3.1}
is an isochronous center.

When the condition (III) is satisfied, introducing the transformation,
$$
u=x (x^2 + y^2)^{-\frac{2}{3}}, \quad v=y (x^2 + y^2)^{-\frac{2}{3}},
$$
into system \eqref{Eq3.1} yields
\begin{equation}\label{Eq3.20}
\begin{array}{ll}
\begin{array}{l}
\dot{u}= - v - \displaystyle\frac{1}{3}\, 
 a_{02} (u^4 + 10 u^2 v^2 - 3v^4) , \\[1.5ex]
\dot{v}= u + \displaystyle\frac{1}{3}\, 
a_{02} (5 u^2 - 7 v^2 )\, u v ,
\end{array}
& (v>0), \\[4.0ex]
\begin{array}{l}
\dot{u}=-v, \\[0.0ex]
\dot{v}=u,
\end{array}
& (v<0).
\end{array}
\end{equation}
The upper-half plane has an inverse integrating factor,
$$
V(u,v)=(u^2 + v^2)^2 f_7^{\frac{5}{6}},
$$
from which a first integral can be obtained as
$$
H(u,v)=\frac{u^2 + v^2}{f_7^{\frac{1}{6}} +16 a_{02} (u^2 + v^2)
\int u f_7^{-\frac{5}{6}}du},
$$
where
$$
f_7=9 \big[ 1 + 2 a_{02}\, v (u^2 - v^2) + a_{02}^2 v^2 (u^2 + v^2 )^2 \big] .
$$
Thus, this system has a transversal system,
\begin{equation}\label{Eq3.21}
\begin{array}{ll}
\begin{array}{l}
\dot{u}=u \big(3-9a_{02}^2v^3 \big)  
\displaystyle\frac{u^2+v^2}{f_7^{\frac{1}{6}}H(u,v)}, \\[0.5ex]
\dot{v}=v \big(3+9 a_{02}u^2v-3a_{02}v^3 \big)
\displaystyle\frac{u^2+v^2}{f_7^{\frac{1}{6}}H(u,v)},
\end{array}
& (v>0), \\[3.0ex]
\begin{array}{l}
\dot{u}=u, \\[0.5ex]
\dot{v}=v,
\end{array}
& (v<0),
\end{array}
\end{equation}
implying that infinity of system \eqref{Eq3.1} is an isochronous center.
\end{proof}

\section{Conclusion}

In this paper, quasi-analytic switching systems have been considered.
A modified and improved method for computing the return maps of
quasi-analytic switching systems is presented.
In particular, a quadratic quasi-analytic switching system is investigated
using this method. Center and isochronous center conditions are
explicitly derived. Compared to the special case $\lambda=1$ for which
five limit cycles are obtained around the origin \cite{GT2003},
we have shown that there exists an infinite number of solutions of 
$\lambda>0,\, \lambda \ne 1$
for the existence of six limit cycles around the 
origin, and two solutions for the existence of 
seven limit cycles for $\lambda = -\frac{16}{5}$,
which bifurcate in the small neighborhood of infinity of the system.
This shows that the dynamics of quasi-analytic switching systems
are more complex.

\vspace{0.20in} 
\noindent 
{\bf Acknowledgements} \ This research is partially supported by 
the National Nature Science
Foundation of China (NSFC No.~11371373 and No.~11601212)
and Applied Mathematics Enhancement Program of Linyi University,
and the Natural Science and Engineering Research Council of Canada
(NSERC No.~R2686A02).




\begin{thebibliography}{99}

\bibitem{Amelikin1982}
Amelikin B, Lukashivich H, Sadovski A.
Nonlinear Oscillations in Second Order Systems.
Minsk, BGY Lenin, B.~I. Press, 1982 (in Russian).

\vspace{-0.05in} 
\bibitem{Bautin1952}
Bautin N~N. 
On the number of limit cycles which appear with the 
variation of coefficients from an equilibrium position of focus or 
center type. 
Mat. Sbornik, 1952, 30: 181--196; 
Amer. Math. Soc. Transl., 1954, 100: 1--19

\vspace{-0.05in} 
\bibitem{Chen2010}
Chen X, Du Z.
Limit cycles bifurcate from centers of discontinuous quadratic systems.
Comp. Math. Appl., 2010, 59: 3836--3848

\vspace{-0.05in} 
\bibitem{Du2005}
Du Z, Zhang W.
Melnikov method for homoclinic bifurcation in nonlinear impact oscillators,
Comp. Math. Appl., 2005, 50: 445--458

\vspace{-0.05in} 
\bibitem{Dulac1908}
Dulac H. 
Dtermination et integration d'une certaine classe 
d'quations diffrentielle ayant par point singulier un centre. 
Bull. Sci. Math. Ser. (2), 1908, 32: 230--252

\vspace{-0.05in} 
\bibitem{Freire2005}
Freire E, Ponce E, Ros J.
The focus-center-limit cycle bifurcation in symmetric 3D piecewise 
linear systems.
SIAM. J. Appl. Math., 2005, 65: 1933--1951 

\vspace{-0.05in} 
\bibitem{Friederiksson2000}
Friederiksson M.~H., Nordmark A.~B.,
On normal form calculations in impact oscillators.
Proc. R. Soc. Lond. Ser. A, 2000, 456: 315--329 

\vspace{-0.05in} 
\bibitem{GT2003}
Gasull A, Torregrosa J.
Center-focus problem for discontinuous planar differential equations.
Int. J. Bifurcation and Chaos, 2003, 13: 1755--1765 

\vspace{-0.05in} 
\bibitem{Hauer2012}
Hauer B, Engelhardt A~P, Taubner T.
Quasi-analytical model for scattering infrared near-field microscopy 
on layered systems.
Optics Expr., 2012, 12: 13173--13188 

\vspace{-0.05in} 
\bibitem{Kapteyn1911}
Kapteyn W. 
On the midpoints of integral curves of differential 
equations of the first degree.
Nederl. Akad. Wetensch. Verslag. Afd. Natuurk. Konikl.
Nederland, 1911, 2: 1446--1457 (in Dutch)

\vspace{-0.05in} 
\bibitem{Kapteyn1912}
Kapteyn W. 
Newinvestigations on the midpoints of integrals of 
differential equations of the first degree. 
Nederl. Akad. Wetensch. 
Verslag Afd. Natuurk., 1912, 20: 1354--1365 (in Dutch)

\vspace{-0.05in} 
\bibitem{Kukuka2007}
Kukuka P.
Melnikov method for discontinuous planar systems.
Nonlinear Anal. Ser. A, 2007, 66: 2698--2719

\vspace{-0.05in} 
\bibitem{Leine2004}
Leine R~I, Nijmeijer H.
Dynamics and Bifurcations of Nonsmooth Mechanical Systems.
in: Lecture Notes in Applied and Computational Mechanics, vol. 18, 
Berlin: Springer-Verlag, 2004

\vspace{-0.05in} 
\bibitem{Li2015}
Li F, Yu P, Tian Y, Liu Y.
Center and isochronous center conditions for switching systems associated 
with elementary singular points.
Commun. Nonl. Sci. Numer. Simulat., 2015, 28: 81--97 

\vspace{-0.05in} 
\bibitem{LL2012}
Liu Y, Li J.
Center and isochronous center problems for quasi analytic systems.
Acta Math. Sinica, 2012, 24: 1569--1582 

\vspace{-0.05in} 
\bibitem{Liu2002}
Liu Y.
The generalized focal values and bifurcations of limit circles for 
quasi-quadratic system.
Acta. Math. Sinica, 2002, 45: 671--682 

\vspace{-0.05in} 
\bibitem{LLH2009}
Liu Y, Li J, Huang W.
Singular point vaules,center problem and bifurcations os limit circles 
of two dimensional differential autonomous systems.
Beijing, Science Press, 2009, 162--190 

\vspace{-0.05in} 
\bibitem{Llibre2014}
Llibre J, Lopes D, Moraes R.
Limit cycles for a class of continuous and discontinuous cubic 
polynomial differential systems.
Qual. Theory Dyn. Syst., 2014, 13: 129--148 

\vspace{-0.05in} 
\bibitem{LM2014}
Llibre J, Mereu C.
Limit cycles for discontinuous quadratic differential systems with two zones.
J. Math. Anal. Appl., 2014, 413: 763--775 

\vspace{-0.05in} 
\bibitem{LV2009a}
Llibre J, Valls C.
Classification of the centers, their cyclicity and isochronicity for 
a class of polynomial differential systems generalizing the linear 
systems with cubic homogeneous nonlinearities.
J. Differential Equations, 2009, 246: 2192--2204

\vspace{-0.05in} 
\bibitem{LV2009b}
Llibre J, Valls C.
Classification of the centers and isochronous centers for a class 
of quartic-like systems.
Nonlinear Anal., 2009, 71: 3119--3128

\vspace{-0.05in} 
\bibitem{LV2009c}
Llibre J, Valls C.
Classification of the centers, their cyclicity and isochronicity 
for the generalized quadratic polynomial differential systems.
J. Math. Anal. Appl., 2009, 357: 427--437

\vspace{-0.05in} 
\bibitem{LV2009d}
Llibre J, Valls C.
Classification of the centers, of their cyclicity and isochronicity 
for two classes of generalized quintic polynomial differential systems.
Nonlinear Differ. Eqn. Appl., 2009, 16: 657--679

\vspace{-0.05in} 
\bibitem{Loud1999}
Lloyd N~G, Pearson J~M. 
Bifurcation of limit cycles and integrability 
of planar dynamical systems in complex form. 
J. Phys. A: Math. Gen., 1999, 32: 1973--1984

\vspace{-0.05in} 
\bibitem{Loud1964}
Loud W~S. 
Behavior of the period of solutions of certain plane 
autonomous systems near centres.
Contrib. Differential Equations, 1964, 3: 323--336

\vspace{-0.05in} 
\bibitem{Pavlovskaia2007}
Pavlovskaia E, Wiercigroch M.
Low-dimensional maps for piecewise smooth oscillators.
J. Sound Vib., 2007, 305: 750--771 

\vspace{-0.05in} 
\bibitem{Pleshkan1969}
Pleshkan I. 
A new method for investigating the isochronicity of a 
system of two differential equations. 
Differential Equations, 1969, 5: 796--802

\vspace{-0.05in} 
\bibitem{Musolino2012}
Musolino A, Rizzo R, Tripodi E.
A Quasi-Analytical Model for Remote Field Eddy Current Inspection.
Pro. Elec. Rese., 2012, 26: 237--249 

\vspace{-0.05in} 
\bibitem{Oerlemans1999}
Oerlemans J.
A quasi-analytical ice-sheet model for climate studies.
Nonlinear Pro. Geop., 1999, 10: 441--452 

\vspace{-0.05in} 
\bibitem{Sabatini1992}
Sabatini M.
Characterizing isochronous centers by Lie brackets.
Diff. Equa. Dyna. Syst., 1997, 5: 91--99 

\vspace{-0.05in} 
\bibitem{Yu2015}
Tian Y, Yu P.
Center conditions in a switching Bautin system.
J. Differential Equations, 2015, 259: 1203--1226 

\vspace{-0.05in} 
\bibitem{Vulpe1988}
Vulpe N~I, Sibirskii K~S. 
Centro-affine invariant conditions for 
the existence of a center of a differential system with cubic nonlinearities.
Dokl. Akad. Nauk SSSR, 1988, 301: 1297-1301. 
Translation in Soviet Math. Dokl., 1989, 38: 198--201 (in Russian)

\vspace{-0.05in} 
\bibitem{P2005}
Xiao P.
Critical point quantities and integrability conditions for
complex planar resonant polynomial differential systems.
Central South University, Ph.D. Thesis, 2005  

\vspace{-0.05in} 
\bibitem{YuHan2012}
Yu P, Han M. 
Four limit cycles from perturbing quadratic integrable systems by 
quadratic polynomials. 
Int. J. Bifurcation and Chaos, 2015, 22: 1250254 (28 pages)

\vspace{-0.05in} 
\bibitem{Zoladek1994}
\.Zo{\l}\c{a}dek H. 
Quadratic systems with center and their perturbations.
J. Differential Equations, 1994, 109: 223--273

\vspace{-0.05in} 
\bibitem{Zou2006}
Zou Y, Kpper T, Beyn W-J.
Generalized Hopf bifurcation for planar Filippov systems continuous 
at the origin.
J. Nonlinear Sci., 2006, 16: 159--177

\end{thebibliography}



\end{document}